\documentclass[11pt]{amsart}
\usepackage{amsmath,amssymb, xypic,verbatim, amscd,color}
\usepackage[margin=0.9in]{geometry}
\usepackage{graphicx}
\usepackage{colordvi}
\usepackage{mathdots}
\usepackage[colorlinks=true, pdfstartview=FitV,
linkcolor=cyan, citecolor=red, urlcolor=blue]{hyperref}
\usepackage{amsfonts,latexsym}

\numberwithin{equation}{section}

\usepackage{amsfonts,amsmath,url, amsthm,xypic,amssymb,graphicx,color,enumerate,mathrsfs, amscd}
\usepackage{bm}

\usepackage{enumerate,bbm}
\usepackage[english]{babel}

\usepackage{ stmaryrd }
\usepackage{ulem}
\usepackage[all]{xy}

\usepackage{thmtools}
\usepackage{thm-restate}

\usepackage{hyperref}

\def\Gm{\mathbb{G}_m}
\def\R{\mathbb{R}}
\def\Q{\mathbb{Q}}
\def\P{\mathbb{P}}
\def\C{\mathbb{C}}
\def\A{\mathbb{A}}
\def\Z{\mathbb{Z}}

\def\Cs{\mathbb{C}^\times}
\def\x{\times}
\def\ox{\otimes}

\def\la{\lambda}

\def\xx{\star}

\newcommand{\gxg}[1]{\left(\sqrt{#1} \ltimes C_B(T) \right)^{\times 2}}

\def\tlt0{(\tilde LT/T)_0}

\def\cU{{\mathcal{U}}}
\def\ft{\mathfrak{t}}

\def\ftz{\mathfrak{t}_\mathbb{Z}}
\def\ftdz{\mathfrak{t}^\vee_\mathbb{Z}}

\def\lr{\mathbb{G}_m^{rot}}

\DeclareMathOperator{\ec}{Spec }

\newcommand{\abcd}[4]{\left(\begin{array}{cc}
  #1 & #2 \\
  #3 & #4 \\
\end{array} \right)}

\newcommand{\ab}[2]{\left(\begin{array}{c}
  #1 \\
  #2 \\
\end{array} \right)}

\newcommand{\mc}[1]{\mathcal{#1}}
\newcommand{\ol}[1]{\overline{#1}}

\newtheorem{thm}{Theorem}[section]
\newtheorem{prop}[thm]{Proposition}
\newtheorem{lemma}[thm]{Lemma}
\newtheorem{cor}[thm]{Corollary}

\theoremstyle{definition}
\newtheorem{definition}[thm]{Definition}
\theoremstyle{remark}
\newtheorem{rmk}{Remark}
\theoremstyle{notation}

\newtheorem{example}{Example}

\usepackage{pgf}
\usepackage{tikz}
\usetikzlibrary{arrows,automata} 
\begin{document}

 \title{Infinite type toric varieties and Voronoi Tilings}
\author{ Pablo Solis}
\address{Department of Mathematics,
 Caltech, 1200 E California Blvd,
Pasadena, CA 91125}
\email{pablos.inbox@gmail.com}

\thanks{I would like to thank my Ph.D advisor Constantin Teleman for suggesting the problem that ultimately lead to this paper. I would also like to thank Melody Chan and Bernd Sturmfels for telling me about Voronoi tilings and Xinwen Zhu for valuable input on the algebraic theory of loop groups. I thank Michel Van Garrel and Bumsig Kim for their support during my visit to KIAS where part of this paper was completed. Finally, the present form of the results here have benefited from the conversations, encouragement and generosity of Dan Halpern-Leistner, Michael Thaddeus, Johan Martens, Tom Graber, Chris Manon, Nadejda Blagorodnova, Lucia Solis and Isreal Zeteno. }

\maketitle

\tableofcontents

\section{Introduction}
An infinite type toric variety is a normal toric variety given by a fan with infinitely many cones. We construct examples in this paper coming from representation theory of loop groups. In fact this construction is a special case of the results presented in \cite{Solis} where an analogue of the wonderful compactification of the loop group of a simple group is constructed. The approach also works for loop groups of tori and that is what we present here.  

Moreover, for the loop group of a torus $T$ with Lie algebra $\ft$, the construction can be expressed using only the data of a nondegenerate symmetric bilinear form $B \colon \ftz \x \ftz \to \Z$ and the choice of a central extension 
\[
1 \to \Gm \to C_B(T) \to T \x \ftz \to 1.
\]
From this data we construct an infinite type toric variety $X_{B,H}$ with a particularly nice fan.

The main result is that the form $B$ together with the lattice of co-characters $\ftz$ determine a Voronoi tiling of $\ft_\R$ and the fan of $X_{B,H}$ is given by the cone on the Voronoi tiling. We expect this construction has a relation to Alexeev and Nakamura's work on degeneration of Abelian varieties to toric varieties as well as to log and tropical geometry.

In the special case of $T = \Gm$ the compactification constructed here recovers the universal cover $C\to \A^1$ of the Tate curve; the generic fiber of $C$ is $\Gm$ and the special fiber $C_0 = \cup_{i \in \Z} \P^1$ is an infinite chain of projective lines. The total space $C$ is an infinite type toric variety and its fan is drawn in figure \ref{fig:rank1}. This curve provides a local model used in \cite{Tolland} to construct gauged Gromov-Witten invariants for the stack $pt/\Gm$. For higher rank tori we do not have a modular interpretation yet for the compactification in terms of bundles on curves but the spaces $X_{B,H}$ do appear to be higher rank versions of the local model used in \cite{Tolland}.

Another connection is with torus orbits in flag varieties as described in \cite{TorbMR1105698}. A special case of \cite[thm\;1]{TorbMR1105698} says that the closure of a generic torus orbit in $G/B$ is the toric variety whose fan is given by the Weyl chamber decomposition of $\ft_\R$. The same toric variety is obtained by taking the closure of the torus inside the wonderful compactification of adjoint group $G_{ad}$. We show this relationship breaks down in the affine case. In general a generic torus orbit closure in an affine flag variety is much smaller than the closure in the wonderful compactification constructed in \cite{Solis}. The latter does have the property that is fan is determined by the Weyl alcove decomposition of $\ft_\R$. One can obtain the latter from the former by taking a limit of orbit closures for increasingly generic points. See specifically theorem  \ref{thm:Torbs}.

Section 2 briefly recalls the defintions of Voronoi and Delaunay tilings. Section 3 defines the central extension $C_B(T)$ and a related semidirect product $\lr \ltimes C_B(T)$ and constructs representations used to define the infinite type toric varieties. The main result is theorem \ref{thm:main}. Section 4 explains the connection with loop groups. Section 5 explains the connection between the wonderful compactification and generic torus orbit closures in flag varieties.

\section{Voronoi and Delaunay}\label{s:Voronoi}
The material in this section largely follows \cite{MR1707764}. In subsequent sections all lattices and their duals come from the characters and co-characters of a torus $T$. Anticipating this application let $(\ftz,\ftdz)$ be a pair of dual lattices; $\ft_\R$ will be the associated $\R$ vector space and $B(x,y)$ will be an inner product on $\ft_\R$; $|x|^2:= B(x,x)$. The convex hull of $S \subset \ft_\R$ is denoted $conv(S)$.

For $x \in \ft_\R$ we say $\lambda \in \ftz$ is an {\it $x$-station} if $|x - \lambda| = \min_{\lambda' \in \ftz} |x - \lambda'|$. Let $sta(x)$ be the set of all $x$-stations; they are the lattice points closest to $x$. A {\it Delaunay cell $\sigma$} is defined as $\sigma = conv(sta(x))=: D(x)$. In general different point $x\neq x'$ can give $D(x) = D(x')$.

The {\it Voronoi cell $V(\sigma)$} associated to $\sigma$ is $\ol{\{ x\in \ft_\R| \sigma = D(x)\}}$. By definition $D(x)$ is constant if $x$ is in the interior $V(\sigma)^0$ and we denote it as $D(\hat{\sigma})$.

\begin{example}\label{ex:2-1-12}
Let $B$ be the inner product on $\R^2$ given by $\abcd{2}{-1}{-1}{2}$. Then the Dalaunay tiling is given by triangles and the Voronoi tiling is hexagonal.
\begin{figure}[htm]
\includegraphics[scale=0.1]{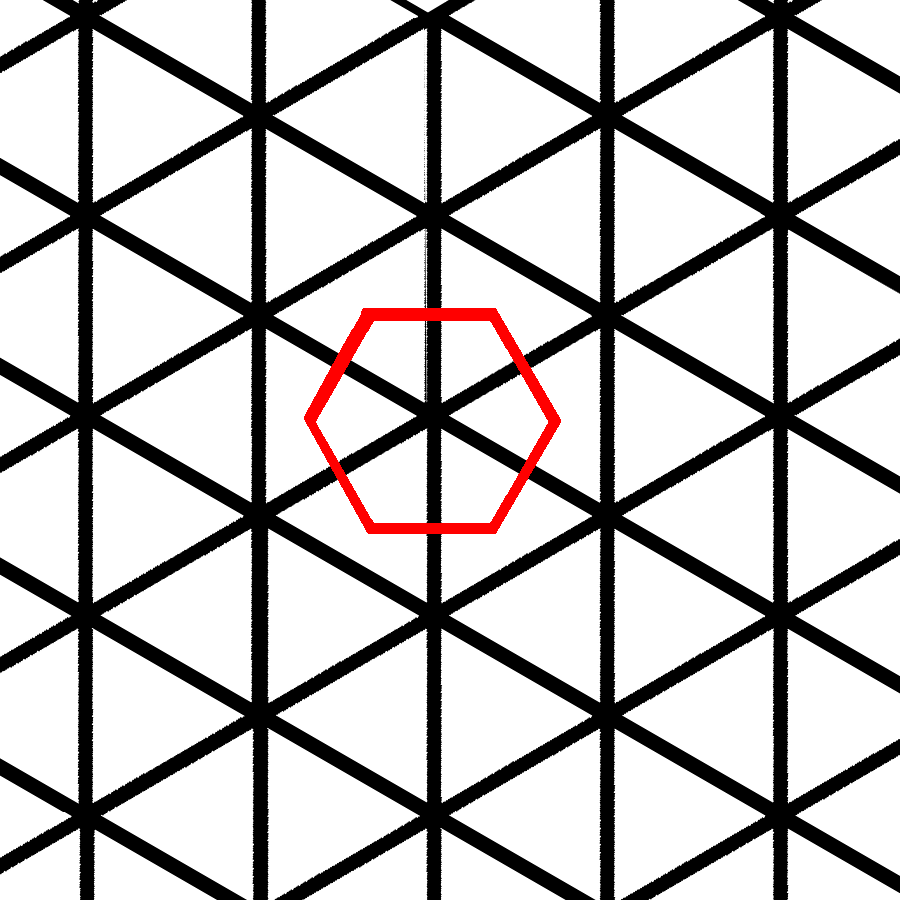}
\includegraphics[scale=0.1]{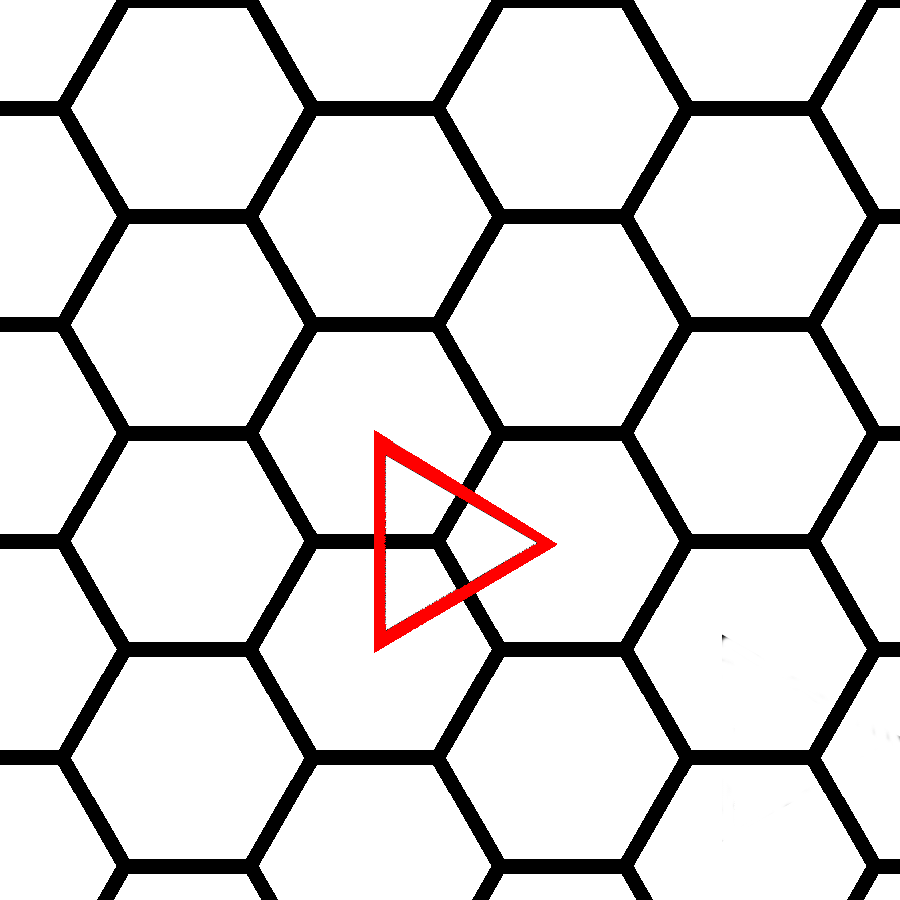}
\caption{Left: the Delaunay tiling with a dual cell marked in red. Right: Voronoi tiling with a dual cell marked in red.}
\label{fig:hex tri}
\end{figure}
\end{example}

The following result is taken from \cite{MR1707764}.
\begin{prop}
There is a 1-to-1 correspondence between Delaunay and Voronoi cells given by $V(\sigma) = \hat{\sigma}$ and $\sigma= D(\hat{\sigma})$ and $\dim \sigma + \dim \hat{\sigma} = \dim \ft_\R$. 
\end{prop}

\section{Toric Central Extenstions}
Let $T = \Gm^r$ be a torus and $(\ftz,\ftdz)$ its group of co-characters and characters. We are interested in central extensions of $T \x \ftz$. In general they are classified by a symmetric bilinear form $B$ on $\ftz$. Specifically we take $B(,)$ to be an inner product on $\ft_\R$ which is integer valued on $\ftz$. In particular, we consider $B$ as a map $\ftz \to \ftdz$, $\lambda\mapsto \lambda^\vee$ via $\lambda^\vee(x) := B(\lambda,x)$.
  
The basic central extension associated to $B$ is denoted $C_B(T)$:
\[
1 \to \Gm \to C_B(T) \to T \x \ftz \to 1
\]
As a set $C_B(T) = \Gm \x(T \x \ftz)$; the symbol $\xx$ will be used to denote multiplication in $C_B(T)$. If $g_i \in T \x \ftz$ then the group structure is \[
(w_1,g_1)\xx(w_2,g_2) = (w_1w_2c_B(g_1,g_2), g_1g_2),
\] 
where $c_B$ is a cocycle given by
\begin{equation}\label{eq:alg central extension}
c_B\left(\ab{t_1}{\lambda_1},\ab{t_2}{\lambda_2}\right) =\lambda_2^\vee(t_1)
\end{equation}
When it is clear that we are working in $C_B(T)$ we abbreviate
\[
\widetilde{u} \leftrightarrow \left(u,\ab{1}{0}\right) \ \ \  t \leftrightarrow \left(1,\ab{t}{0}\right) \ \ \ \lambda \leftrightarrow \left(1,\ab{1}{\lambda}\right).
\]
We will also use the `double' of this extension $C_{2B}(T)$ whose cocycle is given by
\begin{equation}\label{eq:alg central extension2}
c_{2B}\left(\ab{t_1}{\lambda_1},\ab{t_2}{\lambda_2}\right) = \lambda_2^\vee(t_1)\lambda_1^\vee(t_2)^{-1}
\end{equation}
Isomorphic extensions have the same commutators and for $(t,\lambda) \in T \x \ftz$ we have
\begin{align*}
\lambda \xx t \xx (-\lambda) \xx t^{-1} &= \lambda^\vee(t)^{-1} \in C_B(T)\\
\lambda \xx t \xx (-\lambda) \xx t^{-1} &= \lambda^\vee(t)^{-2} \in C_{2B}(T)
\end{align*} 
thus $C_B(T)$ and $C_{2B}(T)$ are not isomorphic.

In the sequel it will be important to incorporate the action of $\Gm$ on $T\x \ftz$ which scales the domain. Specifically for $\theta \in \Gm$ and $\lambda \in \ftz$ define $\lambda \circ \theta$ to be the morphism $z \mapsto \lambda(\theta z)$. We call this action {\it loop rotation} because of its use in loop group, see section \ref{s:loopGroups}.  To differentiate loop rotation from the central $\Gm$ we denote the former as $\lr$. Then $\lr \ltimes \ftz$ is defined by $\theta \lambda \theta^{-1} = \lambda\circ \theta$.

Loop rotation can be lifted to the central extension. If $a,b \in C_B(T)$ or $C_{2B}(T)$ then lifting the $\lr$ action requires 
\begin{equation}\label{eq:constraint}
\theta a \xx b \theta^{-1} = \theta a \theta^{-1} \xx \theta b \theta^{-1}
\end{equation}
This posses no obstruction for $C_{2B(T)}$ and the lift is automatic:

\[
\theta \left(1, \ab{t}{\lambda}\right) \theta^{-1} = \left(1, \ab{t \lambda(\theta)}{\lambda}\right)
\]
However for $C_B(T)$ the constraint \eqref{eq:constraint} is nontrivial and the lift is more interesting:
\begin{align*}
\theta \left(1, \ab{t}{\lambda}\right) \theta^{-1} &= \left(\theta^{\frac{B(\lambda,\lambda)}{2}}, \ab{t \lambda(\theta)}{\lambda}\right)\\
\theta^{\frac{B(\lambda,\lambda)}{2}}&:= \lambda^\vee(\lambda(\theta))^{\frac{1}{2}}
\end{align*}
Because of the  $\frac{1}{2}$ factor the above prescription defines an action of a double cover $\sqrt{\lr}$ of $\lr$. This issue can be circumvented by assuming $B$ is even valued on $\ftz$. 

\subsection{Representations}\label{ss:C_B(T) reps}
First we construct a representation $V$ of $C_B(T)$ and then enhance this to a representation $H = V\ox H_0$ of $\lr \ltimes C_B(T)$; here $H_0$ is a representation of $\lr$.

Let $V_{\mu^\vee}$ be the 1 dimensional representation of $T$ of weight $\mu^\vee$. Set $V = \bigoplus_{\mu \in \ftz} V_{\mu^\vee}$. Consider $\C$ as the trivial representation of $T$, then we can write $V = \C[B(\ftz)] \ox \C$ where $\mu^\vee \ox 1$ spans $V_{\mu^\vee}$. We also introduce $V_2 = \bigoplus_{\mu \in \ftz}V_{2 \mu^\vee} = \C[2 B(\ftz)] \ox \C$.  

There is a natural action of $\ftz$ on $V,V_2$ by translation 
\[
\mu^\vee \ox 1 \xrightarrow{\lambda} (\mu+\lambda)^\vee \ox 1 \ \ \  \ \ 2\mu^\vee \ox 1 \xrightarrow{\lambda} (2\mu+2\lambda)^\vee \ox 1.
\] 

The vector spaces $V,V_2$ have the structure of $C_B(T),C_{2B}(T)$ representations:

\begin{lemma}\label{l:C_B(T) action}
The following prescription 

\begin{align*}
\mu^\vee \ox 1 &\xrightarrow{ \left(1 ,\ab{t}{\lambda}\right)} (\mu +\lambda)^\vee\ox \mu^\vee(t)\\
2\mu^\vee \ox 1 &\xrightarrow{ \left(1 ,\ab{t}{\lambda}\right)} 2(\mu + \lambda)^\vee\ox \mu^\vee(t)^2\lambda^\vee(t)
\end{align*}

defines an action of $C_B(T)$ on $V$ and $C_{2B}(T)$ on $V$ and $V_2$ respectively  such that the central $\Cs$ acts with weight $1$.
\end{lemma}
\begin{proof}
We have $C_B(T)$ is generated by $t,\lambda$ with commutation relation
\[
 \widetilde{\lambda^\vee(t)^{-1}}\xx t \xx \lambda = \left(1 ,\ab{t}{\lambda}\right) = \lambda \xx t.
\]
We compute
\begin{align*}
\lambda\circ t \colon \mu^\vee \ox 1 \xrightarrow{t} \mu^\vee \ox \mu^\vee(t) &\xrightarrow{\lambda} (\mu +\lambda)^\vee\ox \mu^\vee(t)\\
\\
\widetilde{\lambda^\vee(t)}\circ t \circ \lambda\colon \mu^\vee \ox 1 \xrightarrow{\lambda} (\mu +\lambda)^\vee \ox 1 &\xrightarrow{t } (\mu +\lambda)^\vee\ox \mu^\vee(t)\lambda^\vee(t) \\
 & \xrightarrow{\widetilde{\lambda^\vee(t)^{-1}}} (\mu +\lambda)^\vee\ox \mu^\vee(t)
\end{align*}
Hence the operators obey the commutation rule. The proof for $C_{2B}(T)$ is similar.
\end{proof}

We now incorporate the action of $\lr$.  Let $\C_{\theta^i}$ be the irreducible representation of $\lr$ of weight $\theta^i$ and let $H_0 = \C_{\theta^0} \oplus \C_{\theta}$. 

\begin{rmk}\label{rmk:H_0 replace}
For our purposes we can replace $H_0$ with any $\lr$ representation $H'_0 = \oplus_{i \geq 0} H_{\theta^i}$ where $H_{\theta^i}$ is {\it any} finite direct sum of copies of $\C_{\theta^i}$; the only constraint is that $H_{\theta^0},H_{\theta}$ are nonzero. In fact, we can even replace $H'_0$ with $\widehat{H'}_0 = \prod_i H_{\theta_i}$. This occurs when we discuss loop groups in section \ref{s:loopGroups}.
\end{rmk}

Let $H = V\ox H_0$ and $H_2 = V_2 \ox H_0$. Define an action of $\lr$ on $V_2$ by
\begin{equation}\label{eq:lrAction2}
2 \lambda^\vee \ox 1 \xrightarrow{\theta} 2 \lambda^\vee \ox \theta^{B(\lambda,\lambda)} \ \  \ \  \theta^{B(\lambda,\lambda)}:= \lambda^\vee(\lambda(\theta)).
\end{equation}
 and an action of $\sqrt{\lr}$ on $V$ by
 \begin{equation}\label{eq:lrAction}
\lambda^\vee \ox 1 \xrightarrow{\theta} \lambda^\vee \ox \theta^{\frac{B(\lambda,\lambda)}{2}}:= \lambda^\vee(\lambda(\theta))^{\frac{1}{2}}.
\end{equation}

\begin{prop}\label{p:semidirect on H}
The action of $\sqrt{\lr}$ given in \eqref{eq:lrAction} turns $H = V\ox H_0$ into a representation of $\sqrt{\lr} \ltimes C_B(T)$. Moreover each weight space $H_{\theta^i}$ is finite dimensional. Similarly, the action of $\lr$ given in \eqref{eq:lrAction2} turns $H_2 = V_2 \ox H_0$ into a representation of $\lr \ltimes C_{2B}(T)$ and each weight space $H_{2,\theta^i}$ is finite dimensional.
\end{prop}
\begin{proof}
We check that \eqref{eq:lrAction},\eqref{eq:lrAction2} are the unique actions compatible with the conjugation in $C_B(T),C_{2B}(T)$.

We carry this out in the $C_{2B}(T)$ case; the $C_B(T)$ case is similar. By standard conjugation yoga the action of $\theta$ on $2 \lambda^\vee \ox 1$ is equal to the action of $\theta^{-1} \lambda \theta$ on $0 \ox 1$ which is
\begin{align*}
0 \ox 1 &\xrightarrow{\theta^{-1} \lambda \theta\ =\ \left(1,\ab{\lambda(\theta)}{\lambda} \right)} 2 \lambda^\vee \ox \lambda^\vee(\lambda(\theta))
\end{align*}

For the last statement notice that the $\theta^i$ weight space $H_{0,\theta^i} \subset H_0$ is finite dimensional for every $i$. Then
\[
H_{\theta^i} = \bigoplus_{B(\lambda,\lambda)\leq i} 2\lambda^\vee \ox H_{0, \theta^{i - B(\lambda,\lambda)}}
\]
and there are only finitely many $\lambda$ with $B(\lambda, \lambda)\leq i$.
\end{proof}

\subsection{Embedding}\label{ss:Embed}
Using the representations above we would like to construct some toric varieties as certain orbit closures. The basic recipe is as follows. Let $G$ be an algebraic group and $W$ a $G$ representation. Let $[Id_W]$ denote the class of the identity in $\P End(W)$. The orbit closure we seek is 
\[
X_W:= \ol{G \x G [Id_W]} \subset \P End(W).
\]
The stabilizer of $[Id_W]$ always contains $\Delta(G)$ but typically will contain slightly more. So the dense orbit will be a quotient of $G \cong G \x G/\Delta(G)$.

We carry out this construction in the case $G = \sqrt{\lr} \ltimes C_B(T)$ and $W = H$. The result is a scheme with connected components indexed by $\ftz$ and the connected component of the identity is the toric variety we are interested in. One can also do the same with $G = \lr \ltimes C_{2B}(T)$ and $W = H_2$. The results are similar but the embedded torus is slightly different.

It turns out in the $C_B(T)$ case the embedded torus is $\lr \x T/Z_B$ where $Z_B$ is finite group to be defined shortly. In the $C_{2B}(T)$ case the embedded torus is $\lr \x (T/\pm 1)/Z_{2B}$ where if $T = \ec \C[x_1^\pm, \dotsc, x_r^\pm]$ then $T/\pm 1 = \ec \C[x_1^{\pm 2}, \dotsc, x_r^{\pm 2}]$. 

Aside from the embedded torus, all the other essential features of the representations $H,H_2$ behave in the same way. From now on we focus only on the $C_B(T)$ case and its representation $H$.  

The vector space $End(H)$ is too large, it will suffice to work in the smaller subspace:
\begin{align*}
End^\Delta(H) &:= \prod_{i \geq 0} End(H_{\theta^i})
\end{align*}
Then $Id_H \in End^\Delta(H)$ and $End^\Delta(H)$ is preserved by the left and right action of $\sqrt{\lr} \ltimes C_B(T)$. 
\begin{rmk}
The space $End^\Delta(H)$ is the $\C$ points of an affine schemes. In fact, $End^\Delta(H) = \ec R(H)(\C)$ where 
\[
R(H) = \bigsqcup_{i\geq 0} Sym^*End(H_{\theta^i})
\]
where $\bigsqcup$ represents the infinite co-product in the category of commutative rings.
\end{rmk}

Similarly, we consider
\[
End^\Delta(V) := \prod_{i\geq 0} End(V_{\theta^i}).
\]

\begin{definition}
Let $Id_H \in End^\Delta(H)$ and $Id_V \in End^\Delta(V)$ be the respective identity element. Let $Orb([I_H]), Orb([I_V])$ denote the respective orbits under left and right multiplication by $\sqrt{\lr} \ltimes C_B(T)$ in $\P End^\Delta(H),\P End^\Delta(V)$.

We define
\begin{align*}
X_{B,V} \x \ftz = \overline{Orb([I_V])} \subset \P End^\Delta(V)\\
X_{B,H} \x \ftz = \overline{Orb([I_H])} \subset \P End^\Delta(H)
\end{align*}
\end{definition} 
The definition reflects that the connected components of the orbit closure are indexed by $\ftz$. The schemes $X_{B,V},X_{B,H}$ are the orbit closure under the action of $\left(\sqrt{\lr} \x \Gm \x T\right)^{\x 2}$.

Let $v = 0 \ox 1 \subset 0 \ox \C_{\theta^0}$ be a basis vector for $H_{\theta^0}$. Then $v \ox v^*$ is a point of $End^\Delta(H)$ and $\P_0 End^\Delta(H) := \{v \ox v \neq 0 \}$ is an open subscheme. Set
\[
X_{B,H,0} = \P_0 End^\Delta(H) \cap X_{B,H} 
\]

Then $X_{B,H,0}$ is a partial compactification of torus.
Specifically $B\colon \ftz \to \ftz^\vee$ exponentiates to an isogeny $\exp(B) \colon T \to T^\vee$ and the embedded torus is $T/Z_B$ where $Z_B = \ker(\exp(B))$. 

\begin{example}
If $B = \abcd{2}{-1}{-1}{2}$ and $(a,b) \in T$ then $\exp(B)(a,b) = (\frac{a^2}{b}, \frac{b^2}{a})$ and $Z_B \cong \mu_3$.
\end{example}

 \begin{prop}\label{p:X0}
The scheme $X_{B,H,0}$ is a normal toric variety for the torus $\lr \x T/Z_B$.
 \end{prop}
\begin{proof}
It is routine to verify that $X_{B,H,0} = \ec \C[\mathsf{S}]$ where $\mathsf{S}$ is the semigroup generated by all the weights of $H$. If $(n,\lambda^\vee)$ is a weight then so is $(n, - \lambda^\vee)$ ( where $n$ denotes $\theta^{1/2} \mapsto \theta^{n/2}$). Also if $(n,0)$ is a weight then $n$ must be even. Therefore all differences $(n, \lambda^\vee) - (n, - \lambda^\vee)$ generate the character lattice for $\sqrt{\lr}/\pm 1 = \lr$. Hence all differences with first component zero generate the character lattice $B(\ftz)$ for $T/Z_B$.

It remains to show $X_{B,H,0}$ is normal. But this follows because $\Z\mathsf{S}$ generates the entire character lattices hence $\mathsf{S}$ defines a saturated semigroup.\end{proof}

The compactification $X_{B,H}$ is $\gxg{\lr}$ equivariant. In particular there is a conjugation action of $\ftz$.  For $\mu \in \ftz$ one readily verifies
\begin{equation}\label{eq:torusfixed}
\mu\cdot  v \ox v^* \cdot (-\mu) = v_{\mu^\vee} \ox v_{\mu^\vee}^*
\end{equation}

where $v_{\mu^\vee} = \mu^\vee \ox 1$. In an analogous fashion we can define $\P_{\mu^\vee} End^\Delta(H) = \{v_{\mu^\vee} \ox v_{\mu^\vee}^*\}$ and set
\[
X_{B,H,\mu} = X_{B,H} \cap \P_{\mu^\vee} End^\Delta(H).
\]
The open sub varieties $X_{B,H,\mu}$ give us a cover of $X_{B,H}$:

\begin{prop}\label{p:coverX0}
We have $X_{B,H,\mu} = \mu \cdot X_{B,H,0} \cdot(-\mu) $ and 
\[
X_{B,H} = \bigcup_{\mu \in \ftz} X_{B,H,\mu}.
\]
In particular $X_{B,H}$ is a normal toric variety. 
\end{prop}

To prove the proposition we need a couple of preliminary results. Let $\lambda^\dagger = (n, \lambda) \in \Z \oplus \ftz$ be a one parameter subgroup of $\lr(\C) \x T$. Let $o_H^{\lambda^\dagger} \colon \Cs \to \P End^\Delta(H)$ be the $\lambda^\dagger$ orbit of $Id_H$ and define $o_V^{\lambda^\dagger} \colon \Cs \to \P End^\Delta(V)$ similarly.

\begin{lemma}\label{l:n>0 limits}
The limits $\lim_{s\to 0}o_H^{\lambda^\dagger}(s)$, $\lim_{s\to 0}o_V^{\lambda^\dagger}(s)$ exists if and only if the quadratic function $f_{\lambda^\dagger}(\mu) = \frac{n}{2} B(\mu,\mu) + B(\mu,\lambda)$ has a global minimum on $\ftz$ if and only if $n> 0$.
\end{lemma}
\begin{proof}
Consider the image of $o_H^{\lambda^\dagger}$ as an infinite diagonal matrix in $\P End^\Delta(H)$ with nonzer entries corresponding to the weight spaces $\mu^\vee \ox H_{0,\theta^j}$. The weight of $\mu^\vee \ox H_{0,\theta^j}$ is $wt(\mu,j) = (\frac{1}{2}B(\mu,\mu) + j, \mu^\vee)$.  Then $o_H^{\lambda^\dagger}$ is explicitly,
\begin{align*}
orb_{\lambda^\dagger,H} \colon \Cs &\longrightarrow \P End^\Delta(H)\\
s &\mapsto \prod_{\mu,j} s^{  wt(\mu,j)  \circ \lambda^\dagger }  \\
wt(\mu,j)  \circ \lambda^\dagger :=& \frac{n}{2} B(\mu,\mu) + B( \mu,\lambda) +n j
\end{align*}
It follows that $\eta^\dagger(0)$ exists if and only if the function $(\mu,j) \mapsto wt(\mu,j)  \circ \lambda^\dagger$ has a global  minimum. This happens if and only if $n >  0$ . The difference between $wt(\mu,j)  \circ \lambda^\dagger$ and $f_{\lambda^\dagger}$ is a positive term that doesn't change the existence of a global minimum. The same proof applies to $V$.
\end{proof}

In general for any toric variety we can speak of its fan. The fan is simply the collection of possible limit points under 1 parameter subgroups organized by which one parameter subgroups go to the same limit point. If the variety is normal we can recover it from its fan but otherwise the fan of a non-normal toric variety is of little use. 

The toric variety $X_{B,V}$ is not normal in general but its fan agrees with $X_{B,H}$. We record this for later use:

\begin{cor}\label{c:fan}
Under the inclusion $V \subset H$ we have $o_H^{\lambda^\dagger}(0) =  o_V^{\lambda^\dagger}(0)$. In particular the fan of $X_{B,H}$ agrees with the fan of $X_{B,V}$. Moreover the  torus fixed points in $X_{B,H}$ are exactly the points $v_\mu^\vee \ox v_{\mu^\vee}^*$ which appear in \eqref{eq:torusfixed}.
\end{cor}
\begin{proof}
From the proof of lemma \ref{l:n>0 limits} we see that $o_H^{\lambda^\dagger}(0)$ is necessarily supported on the weight $j = 0$ spaces. This immediately gives $o_H^{\lambda^\dagger}(0) =  o_V^{\lambda^\dagger}(0)$ for every $\lambda^\dagger$ for which the limit exists. This shows they have the same fan. Finally, a torus fixed point is necessarily of the form $o_H^{\lambda^\dagger}(0)$ and among these the only ones that are torus fixed are the ones supported on a single weight space. 
\end{proof}

\begin{proof}[proof of proposition \ref{p:coverX0}]
First we have 

\begin{align*}
X_{B,H,\mu} &= X_{B,H} \cap \P_{\mu}End^\Delta(H)\\
&=  X_{B,H} \cap \mu \cdot \P_0 End^\Delta(H) \cdot (-\mu)\\
&= \mu \cdot X_{B,H,0} \cdot(-\mu).
\end{align*}
Next, consider the subscheme
\[
X_{B,H} - \bigcup_{\mu } X_{B,H,\mu}.
\]
It is closed, torus stable, and contains no torus fixed points hence it is empty. Therefore, using proposition \ref{p:X0}, we conclude that $X_{B,H}$ is a normal toric variety.
\end{proof}

Now we can prove the main theorem.
\begin{thm}\label{thm:main}
The variety $X_{B,H}$ is a normal toric variety for $\lr \x T/Z_B$ whose fan is given by the cone on the Voronoi tiling of $\ft_\R$.
\end{thm}

\begin{proof}
Proposition \ref{p:coverX0} shows $X_{B,H}$ is normal and proposition \ref{p:X0} shows the embedded torus is $\lr \x T/Z_B$. It remains to compute the fan of $X_{B,H}$.  By corollary \ref{c:fan} we can compute the fan for $X_{B,V}$. 

First let us determine the possible limit points. Let $\lambda^\dagger = (n,\lambda)$ with $n>0$. Then by lemma \ref{l:n>0 limits} the limit $o^{\lambda^\dagger}_V(0)$ is controlled by the function $f_{\lambda^\dagger}(\mu) = \frac{n}{2} B(\mu,\mu) + B(\mu,\lambda)$. This function is minimized at $\mu = \frac{-\lambda}{n} \in \ft_\Q$. 

For a weight vector $v_{\mu^\vee} = \mu^\vee \ox 1$ in $V$ let $e_{\mu^\vee} = v_{\mu^\vee} \ox v_{\mu^\vee}^* \in End^\Delta(V)$. To bring in the terminology of section \ref{s:Voronoi} we see that $o^{\lambda^\dagger}_V(0) $ is supported on the $\frac{\lambda}{n}$-stations in $\ftz$; we set $sta(\lambda^\dagger):= sta(\frac{\lambda}{n})$:
\[
o^{\lambda^\dagger}_V(0) = \sum_{\mu \in sta(\lambda^\dagger)} e_{\mu^\vee} 
\]
Moreover if $\lambda^\dagger,\eta^\dagger$ are two 1-parameter subgroups then $o^{\lambda^\dagger}_V(0) = o^{\eta^\dagger}_V(0)$ if and only if $sta(\lambda^\dagger) = sta(\eta^\dagger)$. In other words, $D(\lambda^\dagger) = Conv(sta(\lambda^\dagger))$ is a Delaunay cell in $\ft_\R$ and $o^{\lambda^\dagger}_V(0) = o^{\eta^\dagger}_V(0)$ if and only if $D(\lambda^\dagger) = D(\eta^\dagger)$ which is to say if $\eta^\dagger = (m,\eta)$ then $\frac{\eta}{m}$ is a rational point on the Voronoi cell $V(D(\lambda^\dagger))$ that is dual to $D(\lambda^\dagger)$.  
\end{proof}

\subsection{Examples}
For $T = \Gm$ take $B(n,m) = nm$. Then vertices of the Voronoi tiling of $\ft_\R = \R$ are $\frac{1}{2}+\Z$ and we embed $\ft_R$ in $\R \oplus \ft_R$ as the hyperplane $1 \oplus \ft_\R$. Then $X_{B,H}$ is the toric variety with fan given in figure \ref{fig:rank1}.

\begin{figure}[htm]
\includegraphics[scale=0.3]{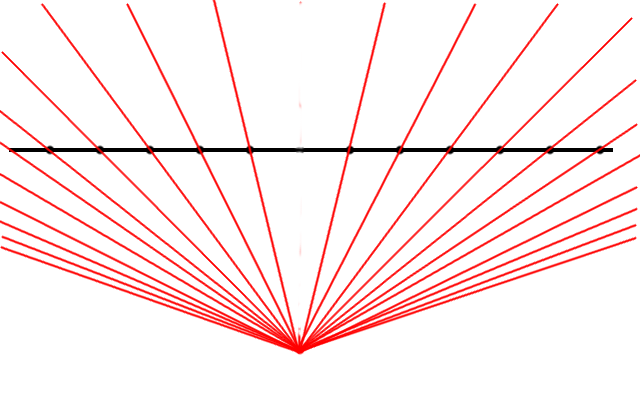}
\caption{Fan for $T = \Gm$. The black line is the hyperplane $1 \oplus \ft_\R$}
\label{fig:rank1}
\end{figure}
 
There is a morphism $X_{B,H} \to \A^1$. The generic fiber if $\Gm$ and the special fiber is an infinite chain of projective lines. 
 
Next take $T=\Gm^2$ and $B$ the inner product given by the matrix $\abcd{2}{-1}{-1}{2}$. Then Voronoi tiling is a hexagonal tiling given in figure \ref{fig:hex tri}. The fan is cone on this hexagonal tiling. It is depicted in figure \ref{fig:3dfan}.

\begin{figure}[htm]
\includegraphics[scale=0.15]{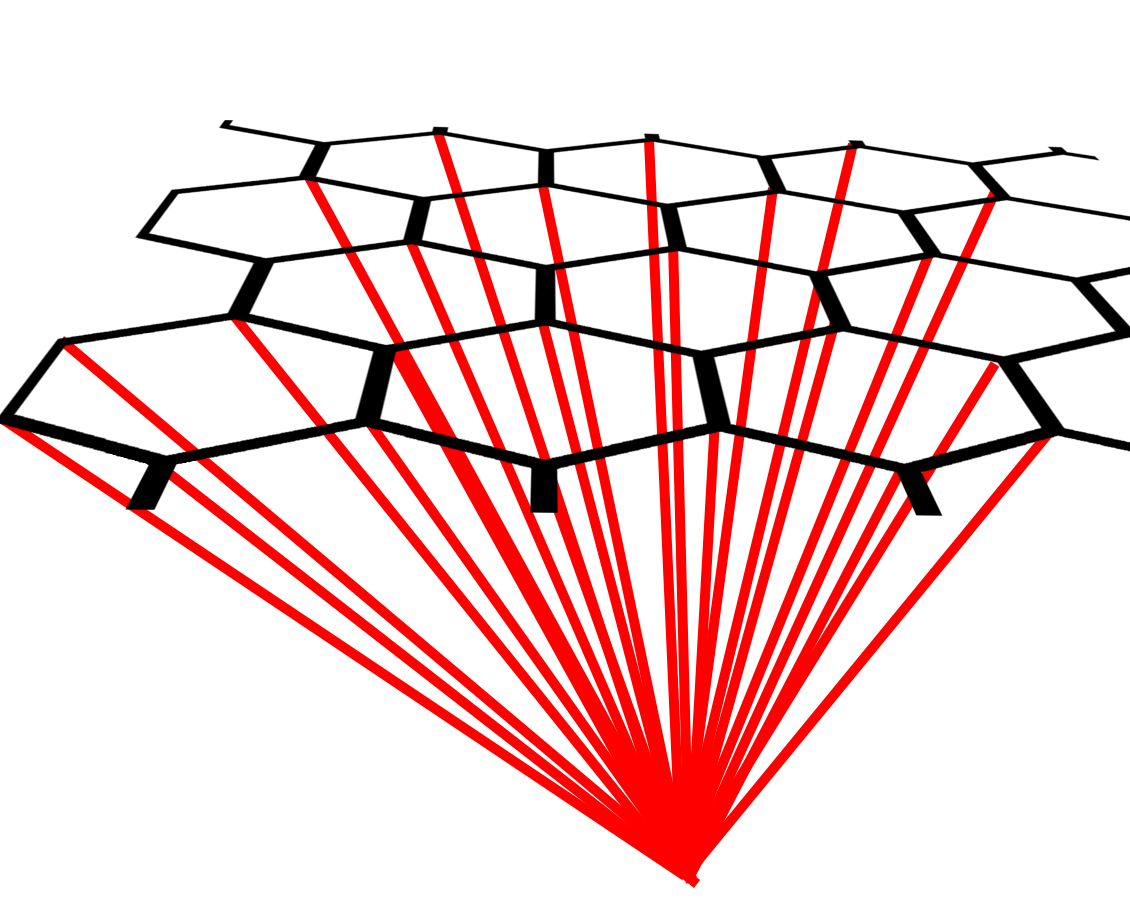}
\caption{The fan for a rank 2 torus. The hexagons show the Voronoi tiling of $1 \oplus \ft_\R$}.
\label{fig:3dfan}
\end{figure}

\section{Loop groups of Tori}\label{s:loopGroups}
The central extensions $C_B(T),C_{2B}(T)$ naturally come from loop groups. We explain the connection. 

Let $\mathbf{Aff}_\C$ denote the category of $\C$-algebras, $\mathbf{Set}$ the category of sets and $\mathbf{Grp}$ the category of groups. Let $G$ be an affine algebraic group over $\C$.
\begin{definition}
The loop group $LG\colon \mathbf{Aff}_\C \to \mathbf{Grp}$ is the functor given by $LG(R):= G(R((z)))$ where $R((z))$ is the ring of formal Laurent series with coefficients in $R$.
\end{definition}
It is known that $LG$ is represented by an ind-scheme; an increasing union of infinite dimensional schemes. Elements $g(z) \in LG(R)$ are called loops. 

There are a few natural subgroups of $LG$.
\begin{itemize}
\item positive loops $L^+G(R):= G(R[[z]])$
\item $\cU := \ker \left(L^+G \xrightarrow{z\mapsto 0} G\right)$
\item negative loops $L^-G(R):= G(R[z^{-1}])$ 
\item $\cU^- := \ker \left(L^-G \xrightarrow{z^{-1}\mapsto 0} G\right)$
\item 1 parameter subgroups $\mathfrak{g}_\Z = \hom_{\mathbf{Grp}}(\Gm,G)$
\end{itemize}
For the last subgroup, let $i \colon \ec R((z)) \to \ec \C[z^\pm]$ be the natural map, then $\mathfrak{g}_\Z \subset LG$ via $\lambda \mapsto \lambda\circ i$.

Now let us specialize to $G =T = (\C^\times)^r$. Then $LT(R) =(R((z))^\times)^r$. We can also describe $\cU,\cU^-$ more explicitly.   
\begin{align*}
\mathbb{W}(R)&=\left\{1+\sum_{i=0}^\infty a_i z^i | a_i \in R\right\}\\
\widehat{\mathbb{W}}(R) &= \left\{1+\sum_{i\in I, |I| < \infty} a_i z^i | a_i \in \sqrt{(0)}\right\}
\end{align*}

\begin{lemma}\label{l:LTfactors}
The loop group $LT$ factors as
\[LT = T\x \ftz \x \cU \x \cU^-\]
and $\cU(R) = \mathbb{W}(R)^{\dim T}$, $\cU^-(R) = \widehat{\mathbb{W}}(R)^{\dim T}$. Moreover $f \in \mathbb{W}(R)$, $g \in \widehat{\mathbb{W}}(R)$ can be uniquely expressed as infinite products
\[
f = \prod_{i\geq 1}(1-r_i z^i)\ \ \ \  g= \prod_{j\geq 1} (1 - r_j z^{-j})
\]
where in the second case only finitely many of the $r_j$ are nonzero. 
\end{lemma}
\begin{proof}
See \cite[2.13]{XinwenNotes}.
\end{proof}

We next discuss a central extension of $LT$ by $\Gm$; these actually arise from central extensions of $T$ by $\mathbf{K}_2$ which we recall: 
\begin{definition}
The Milnor $K$ group $\mathbf{K}_2 \colon \mathbf{Aff}_\C \to \mathbf{Set}$ is the quotient sheaf of $\Gm\ox\Gm$ by the relations $a\ox b + b \ox a$ and $a\ox(1-a)$ for $a \neq 0,1$. 
\end{definition} 

In \cite{MR1896177}, Brylinski and Deligne construct central extensions of reductive groups by $\mathbf{K}_2$ and in particular central extensions
\begin{equation}\label{eqn:K2}
1\to \mathbf{K}_2 \to \widetilde{T} \to T \to 1
\end{equation}
are classified by an integer valued quadratic form $B$ on $\ftz$ and a central extension
\[
1 \to \Gm \to \widetilde{\ftz} \to \ftz \to 1
\]
Satisfying a certain condition ( \cite[eq. (0.5)]{MR1896177}). 

To get from $\mathbf{K}_2$ to $\Gm$ we need a symbol $\{-,-\} \colon L\mathbf{K}_2 \to \Gm$. A famous example is the Contou-Carr{\`e}re given as follows. Represent an element of $L\mathbf{K}_2(R)$ by a pair $(f,g) \in \Gm(R((z)))^{\x 2}$ and using lemma \ref{l:LTfactors} write
\[
f = f_0z^n \prod_{i\geq 1}(1-f_i z^i)\prod_{j\geq 1}(1-f_j z^{-j}) \ \ \ \ g = g_0z^m \prod_{i\geq 1}(1-g_i z^i)\prod_{j\geq 1}(1-g_j z^{-j})
\]
Then the Contou-Carr{\`e}re symbol \cite{MR1272340} is given by 
\[
\{f,g\} = (-1)^{n m}\frac{f_0^m}{g_0^n}\frac{\prod_{i\geq 1}\prod_{j\geq 1}\left(1 - f_i^{j/(i,j)}g_{-j}^{i/(i,j)}\right)^{(i,j)}}{\prod_{i\geq 1}\prod_{j\geq 1}\left(1 - g_i^{j/(i,j)}f_{-j}^{i/(i,j)}\right)^{(i,j)}}.
\]

To obtain the central extension of $LT$ one applies the ``$L$'' functor to \eqref{eqn:K2} and pushes out along the chosen symbol:
\[
\xymatrix{1\ar[r] &L\mathbf{K}_2\ar[r]\ar[d]^{\{-,-\}} & L\widetilde{T}\ar[r]\ar[d] & LT\ar[r]\ar[d]^{=} & 1\\
1\ar[r] &\Gm\ar[r] & \widetilde{LT}\ar[r] & LT\ar[r] & 1
}
\]

However in the sequel we will be primarily interested in $\C$ points and the story simplifies dramatically. First, $\cU^-(\C) = 1$ and, for example, the Contou-Carr{\`e}re symbol collapses to the Tame symbol:
\[
\{f,g\} = (-1)^{n m}\frac{f_0^m}{g_0^n}.
\]
In particular for the extension $\widetilde{LT}$ given by the Contou-Carr{\`e}re symbol we have $\widetilde{LT}(\C) \cong C_{2B}(T) \x \cU$, where $C_{2B}(T)$ was described in \eqref{eq:alg central extension2}.

We briefly remark on the analytic construction of these central extension. These were constructed by Segal in \cite{MR626704}. Analytically one considers $L^{sm}U(1)^r:= C^\infty(S^1,U(1)^n)$. This is heuristically a ``compact real form'' of $LT$, but the only precise statement is that $L^{sm}U(1)^r$ contains a subgroup $L^{sm,poly}U(1)^r$ which is a compact real form of $L^{poly}T(\C) := T(\C[z^\pm]) = T\x \ftz$. 

Identifying $S^1 = \R/\Z$ then any analytic loop $\gamma$ has a logarithm:
\[
\xymatrix{
S^1 \ar[r]^\gamma & U(1)^n \\ \R \ar[r]^{f = \log(\gamma)}\ar[u]^{e^{i}} & \ft_{\R} \ar[u]^{e^i}
}.
\]
A map $f \colon \R \to \ft_\R$ exponentiates if and only if $\Delta_f := \frac{f(x + 2 \pi) - f(x)}{2 \pi}$ is constant and lies in $\ftz \subset \ft_\R$. Let $F_\Z$ denote the space of such $f$; any $\gamma \in L^{sm}U(1)^r$ can be written as $\gamma = \exp(i f)$ where $f \in F_\Z$. 

Recall we have the data of an inner product $B$ on $\ft_\R$. There is a bilinear form on $F_\Z$ given by \cite[pg.\;313]{MR626704}
\begin{equation}\label{eq:an central ext}
S(f,g) = \int_0^{2 \pi} \frac{B(f',g)}{4 \pi} d\theta + \frac{B(\Delta_f, g(0))}{2}
\end{equation}
Then \eqref{eq:an central ext} exponentiates to give the analytic central extension of $L^{sm}U(1)^r$. Finally the loops $U(1)^r \x \ftz \subset L^{sm}U(1)^r$ are presented by affine linear maps $R \to \ft_\R$ and one checks that the restriction of \eqref{eq:an central ext} recovers \eqref{eq:alg central extension} after exponentiatation.  The `double' of $S(f,g)$ is $S'(f,g) = S(f,g) - S(g,f)$ \cite[pg.\;313]{MR626704} and in a similar fashion $S'$ recovers  \eqref{eq:alg central extension2} after exponentiation.

We also have $\lr \ltimes LT$ whose set of points is just the product $\lr(R)\x LT(R)$ but conjugation by $\theta \in \lr(R)$ is $\theta \gamma(z) \theta^{-1} = \gamma( \theta z)$ hence
\[
\ab{\theta_1}{\gamma_1(z)}\cdot \ab{\theta_2}{\gamma_2(z)} := \ab{\theta_1\theta_2}{\gamma_1(\theta_2^{-1}z)\gamma_2(z)}
\]
The action of $\lr$ is called loop rotation. It lifts to the central extension and we finally obtain $\lr \ltimes \widetilde{L T}$.

\subsection{Representation and Embedding}
For the rest of this section we restrict ourselves to $\C$-points. 

Let $W = Lie(\cU(\C)) = z\C[[z]]\ox \ft_\C$. For $X \in W$ let $A(X)\colon Sym^*(W) \to Sym^*(W)$ be the multiplication operator $f \mapsto X f$; this gives a representation of $W$ on $Sym^*(W)$. 

Moreover the exponential map $\exp \colon z\C[[z]]\ox \ft_\C$ has an inverse via the standard formula $\log(1+x) = \sum_{i=1}^\infty \frac{(-1)^{i+1}}{i}x^i$. Hence we obtain a representation of $\cU(\C)$ on the completed symmetric product $H_0 := \widehat{Sym^*(W)}$ via $\gamma \mapsto \exp(A(\log(\gamma))$. 

Parallel to the discussion in \ref{ss:C_B(T) reps}, we obtain a representation of $\widetilde{LT}(\C)$ on 
\[
H = \bigoplus_{\lambda \in \ftz} H_{\lambda^\vee}
\]
With actions of $\cU(\C), T(\C), \ftz$ given by: 
\begin{itemize}
\item $H_{\lambda^\vee} = H_0$ as $\cU(\C)$  representations,
\item $T(\C)$ act with weight $\lambda^\vee$ on $H_{ \lambda^\vee}$,
\item $\mu \in \ftz$ acts by the `identity' $\mu \colon H_{\lambda^\vee} \to H_{\lambda^\vee + \mu^\vee}$,
\item the central $\Cs$ acts with weight $1$ on $H$.
\end{itemize}

Denote by $B(\ftz)$ the image of $\ftz$ in $\ftz^\vee$. Then we can equivalently describe $H = \C[B(\ftz)]\ox H_0$ where we identify $\lambda^\vee\ox H_0$ with $H_{\lambda^\vee}$. The sub vector space spanned by all $\lambda^\vee\ox 1$ is denoted $V = \C[B(\ftz)]\ox \C$. 

\begin{rmk}
Each $\ol{\nu} \in \ftz^\vee/B(\ftz)$ gives rise to another representation by replacing $H_0$ with $H_{\nu}$, however we will not need these other representations. 
\end{rmk}

We would like to consider the $\lr(\C) \ltimes \widetilde{LT}(\C)^{\x 2}$ orbit of the identity in $\P End(H)$. This can be accomplished exactly as in section \ref{ss:Embed} using the decomposition of $H$ under loop rotation $H = \bigoplus_{i} H_{\theta^i}$. However one modification is necessary. The action of $\lr(\C) \ltimes \widetilde{LT}(\C)^{\x 2}$ does not preserve $End^\Delta(H)$. One needs a slightly bigger space:
\begin{equation*}
End^+(H) := \prod_{i\leq j} Hom(H_{\theta^i}, H_{\theta^j})
\end{equation*}
Then $Id_H \in End^\Delta(H) \subset End^+(H)$ and $End^+(H)$ is preserved by the left and right action of $\lr(\C) \ltimes \widetilde{LT}(\C)$.

One can now proceed as in section \ref{ss:Embed} and look at the orbit closure of the identity. The development in section \ref{ss:Embed} did not use loop group because everything ultimately reduces to toric data. Specifically the orbit closure in the loop group case gives a compactification of the form $X_{B,H} \x \ftz \x \cU(\C)$.

The appearance of Voronoi and Delaunay tilings appear in many other places. For example they appear in the study of Berkovich spaces. Some of these connections maybe superficial but there is at least one that seems deeper. Namely the connection with Alexeev and Nakamura's work \cite{MR1707764}on degeneration of Abelian varieties.

The construction of $X_{H,B}$ essentially comes from the representation theory of $L T$. The original motivation to work with loop groups was due to their connection with the moduli space of bundles on a curve. Specifically the partial compactifications obtained here are expected to give degenerations to nodal curves of the moduli space of $T$ bundles on a smooth curve.

This is certainly true for in the rank 1 case, see specifically \cite{Tolland}. The moduli space of $T$ bundles is essentially a product of Jacobians hence the connection with degenerations of Abelian varieties.

\section{Torus Orbits}\label{s:TorusOrbits}
We begin by reveiwing the finite dimensional story. 

Let $G$ be a semisimple group and let $Z(G)$ denote the center of $G$ and $G_{ad} = G/Z(G)$.  We denote by $\ol{G_{ad}}$ the wonderful compactification of $G_{ad}$ first constructed by De Concini and Procesi \cite{MR718125}.

The variety $\ol{G_{ad}}$ is a smooth $G \x G$ equivariant compactification whose boundary is a smooth normal crossing divisor. It has a unique closed orbit isomorphic to $G/B \x G/B$ and one of the interesting features of $\ol{G_{ad}}$ is that it is a projective variety that interpolates between $G_{ad}$ and $G/B \x G/B$. It is a spherical variety (the anlouge of toric varieties for reductive groups) and is moreover a toroidal spherical variety. The closure $\ol{T_{ad}}$ of $T/Z(G) \subset G$ is naturally an important object and known to the toric variety whose fan is given by the Weyl chamber decomposition of $\ft_\R$.

On the other hand for any parabolic subgroup $P$ we have a projective variety $G/P$ and for any $p \in G/P$ we obtain a projective variety as the closure $\ol{T\cdot p}$. As $p$ varies the varieties $\ol{T\cdot p}$ vary in dimension. There is an open set $U$ of generic points such that $dim T\cdot p = \dim T$ and the toric variety $\ol{T\cdot p}$ is independent of the choice of $p \in U$.

In \cite{TorbMR1105698} a definition of a generic torus orbit is given in terms of intersections of various open cells in $G/P$. We give an different definition of generic point which more suitably generalized to the affine case. Potentially replacing $P$ with a conjugate $g P g^{-1}$ there is an irreducible highest weight representation $V(\lambda)$ such that $P=Stab([v_\lambda]) $, the stabilizer of the highest weight in $\P V(\lambda)$. Let $V(\lambda)= \bigoplus_{\chi \in \ft^\vee_\Z} V_\chi$. We say $p \in G/P \subset \P V(\lambda)$ is generic if any lift of $p$ to $V(\lambda)$ satisfies that the projection to each $V_\chi$ lies in $V_\chi - 0$. This condition is stronger than the one given in \cite{TorbMR1105698} in the sense that any generic point of in the sense just described is generic in the sense of \cite{TorbMR1105698} but the converse may not be true.

By \cite[thm\;1]{TorbMR1105698} it follows that the closure of a generic $T$ orbit in $G/B$ is a toric variety with torus $T_{ad}$ and fan given by the Weyl chamber decomposition of $\ft_\R$. The naive affine generalization of this story would take the closure of a maximal torus in an affine analogue of the wonderful compactification and compare it with the orbit closure of a generic torus orbit in the affine flag manifold $LG/\mc{B}$; here $\mc{B} = \{g \in L^+G| g(0) \in B\}$. 

The affine analogue $X^{aff}$ of the wonderful compactification for $LG$ has been constructed in \cite{Solis}. It is also proved in \cite{Solis} that the closure $Y^{wond}$ of a maximal torus in $X^{aff}$ is the cone on the Weyl alcove decomposition; this much of the generalization holds. 

Once we give a notion of generic torus orbit in $LG/\mc{B}$ one can compare $Y^{wond}$ with a toric variety $Y^{flag} \subset LG/\mc{B}$. However $Y^{wond}$ does not agree with $Y^{flag}$; the reason is any point of $LG/\mc{B}$ (generic or otherwise) lies in a finite dimensional projective variety. Therefore any generic torus orbit closure in $LG/\mc{B}$ is of finite type whereas $Y^{wond}$ is an infinite type toric variety. 

Nevertheless we can prove a relationship between $Y^{wond}$ and $Y^{flag}$. We begin with the definition of a generic torus orbit in $LG/\mc{B}$.  Let $\underline{\lambda}$ denote a regular dominant highest weight of $L G$ and let $\P V(\underline{\la})$ be the corresponding projective space on which $L G$ acts. It is known that this representation extends include loop rotation $\Gm^{rot} \ltimes LG$ and $\Gm^{rot} \x T$ is a maximal torus. Moreover the action of $\Gm^{rot}$ lifts to $V(\underline{\la})$ and gives a decomposition
\begin{equation}\label{eq:decomp}
V(\underline{\la}) = \bigoplus_{i\geq 0} V_i = \bigcup_{i \geq 0} V_{\leq i}
\end{equation}
where $\Gm^{rot}$ acts by the $i$th power on $V_i$ and each $V_i$ is finite dimensional. The orbit of the highest weight gives a projective embedding $LG/\mc{B} \subset \P V(\underline{\la}) = \cup_i \P V_{\leq i}$.

\begin{definition}
Suppose under the projective embedding of $LG/\mc{B}$ we have $p\in LG/\mc{B} \cap \P V_i$. Let $V_{\leq i}= \bigoplus_{\chi \in \ft^\vee_\Z} V_{\leq i,\chi}$ be its decomposition under $T\subset G$. We say $p$ is $i$-generic if any lift of $p$ to $V_{\leq i}$ satisfies that the projection to each $V_{\leq i,\chi}$ lies in $V_{\leq i,\chi} - 0$. Moreover define $Y^{flag,i}$ to be the closure of $\Gm^{rot}\x T$ orbit of an $i$-generic point.
\end{definition}

\begin{thm}\label{thm:Torbs}
\begin{itemize}Let $i\geq 1$.
\item[(1)] The projective toric variety $Y^{flag,i}$ contains $\Gm^{rot} \x T/Z(G)$ as a dense open subvariety. 
\item[(2)] There is an open toric subvariety $Y^{flag,i}_0$ which is normal and whose fan is the cone on a finite union of Weyl alcoves.
\item[(3)] $Y^{wond} = \bigcup_{i\geq 1} Y^{flag,i}_0$.
\end{itemize} 
\end{thm}

\begin{proof}
Statement (1) follows from the fact that the lattice generated by the weights of $V(\underline{\la})$ is the root lattice of $Lie(\Gm^{rot} \ltimes LG)$ and this root lattice is the character lattice of the adjoint torus $\Gm^{rot} \x T/Z(G)$. 

For (2) we briefly recall the construction of $Y^{wond}$. Namely one takes the closure in $\P(\prod_j End(V_j))$ of the $T$ orbit of the identity. Let $p \in LG/\mc{B} \subset \P V(\underline{\la})$ be an $i$-generic point. If the projection of $y \in Y^{wond} \subset \P(\prod_j End(V_j))$ to $\P (\prod_{j \leq i} V_j)$ is defined then it determined an endomorphism of $\P V_{\leq i}$ which is defined at $p$. Let $Y_i^{wond}$ be the subset of points such that $Y^{wond} - \Gm^{rot}\x T/Z(G) \subset \P V_{\leq i}$. Then $Y_i^{wond}$ is open because its complement is $\partial Y^{wond} \cap \P (\prod_{j> i} End(V_j))$.

By construction $Y^{wond}_i$ acts on $p$ and in particular preserves the $\Gm^{rot} \x T$ orbits closure $Y^{flag,i}$ of $p$. Let $Z \subset Y^{flag,i}$ be the image of $p$ under $Y^{wond}_i$. We have that $Z$ contains the maximal torus and is torus stable thus the complement $Z^c$ of $Z$ has strictly lower dimension and it torus stable. Set $Y^{flag,i}_0 = Y^{flag,i} - \overline{Z^c}$. Then $Y^{flag,i}_0$ is an open and torus stable and because $p$ is $i$-generic we have that $Y^{flag,i}_0$ is identified with an open toric subvariety of $Y^{wond}_i$. Consequently $Y^{flag,i}$ is a normal toric variety whose fan is given by the cone on a union of Weyl alcoves. Because $Y^{flag,i}$ is finite type the union is finite.

Recall the fan of $Y^{wond}$ is the cone of the Weyl alcove decomposition of $\ft_\R$. For any given alcove $A$ the boundary points of $Y^{wond}$ corresponding to $A$ lie in some $\P (\prod_{j\leq i} End(V_j))$. Choose a fundamental alove $A_0$ with $i$ as small as possible. Then all other alcoves are obtained as $w A_0 w^{-1}$ for $w \in W^{aff}$, the affine Weyl group. For any fixed $w$ we have that the boundary points of $w A_0 w^{-1}$ lie in $\P (\prod_{j\leq i} End(V_j))$ for $i$ sufficiently large. In particular, for such $i$ the boundary points of $w A_0 w^{-1}$ will lie in $Y^{flag,i}_0$. It follows that
\[
Y^{wond} = \bigcup_{i\geq 1} Y^{flag,i}_0.
\]  
\end{proof}

As a final remark, in \cite{Solis} the notion of a positive 1 parameter subgroup of $\Gm^{rot} \x T$ was defined. Namely $\Gm \xrightarrow{\lambda} \Gm^{rot} \x T$ is positive if the composition $\Gm \xrightarrow{\lambda} \Gm^{rot} \x T \xrightarrow{pr_1} \Gm^{rot}$ is given by a $t\mapsto t^n$ for $n>0$. It follows that $Y^{flag,i}_0$ consists of $\Gm^{rot} \x T/Z(G)$ together with boundary points that are limits of positive 1 parameter subgroups. But we do not know if every point of $Y^{flag,i}$ which is a limit of positive 1 parameter subgroups lies in $Y^{flag,i}_0$.

We do not know if the varieties $Y^{flag,i}$ are always normal. But the following result in the literature of toric varieties describes its normalization:

\begin{prop}
Let $\{\chi_1, \dotsc, \chi_m\}$ be a finite subset of characters of a torus $T$ such that the differences $\chi_i - \chi_j$ generate $\ft^\vee_\Z$. Let $Y$ be the closure of a generic torus orbit in $\P( \bigoplus_{i = 1}^m \C_{\chi_i})$. Then the normalization $\widetilde{Y}$ of $Y$ is a projective toric variety with polytope $P$ given by the convex hull of $\{\chi_1, \dotsc, \chi_m\}$ in $\ft^\vee_\R$ and moreover $Y$ is normal if and only if for each $\chi_i$ that is a vertex of $P$ the semigroup generated by $\chi_j - \chi_i$ is saturated.
\end{prop}

Applying this proposition to $Y^{flag,i}$ one can say that most of the vertices of $P$ are orbits of the highest weight under the action of $W^{aff}$; these vertices define open affine pieces of $Y^{flag,i}$ which are normal. However choosing an $i$-generic point introduces at cutoff producing vertices of $P$ which are not points in the $W^{aff}$ orbit of the highest weight and this might create a non-normal compactification. 

\bibliographystyle{plain} 

\bibliography{LTfan}

\end{document}